\documentclass{amsart}
\usepackage{amsmath,amsfonts,amssymb,enumerate}

\let\g=\gamma

\newcommand{\N}{\mathbb{N}}                   
\newcommand{\R}{\mathbb{R}}                   
\newcommand{\C}{\mathbb{C}}                   
\newcommand{\K}{\mathcal{K}}                  
\newcommand{\OO}{\mathcal{O}}                 
\newcommand{\VV}{\mathcal{V}}                 
\newcommand{\GG}{\mathcal{G}}                 
\newcommand{\Ad}{\mathcal A^d}                
\newcommand{\mm}{\mathfrak{m}}                
\newcommand{\vr}{\varrho}                     
\newcommand{\ve}{\varepsilon}                 
\newcommand{\kp}{\kappa}                      
\newcommand{\lb}{\lambda}                     
\newcommand{\Hp}{\mathrm{Hol}_p}              
\newcommand{\reg}{\mathrm{reg}}               
\newcommand{\D}{\Delta}

\newtheorem{theorem}{Theorem}[section]
\newtheorem{proposition}[theorem]{Proposition}
\newtheorem{lemma}[theorem]{Lemma}
\newtheorem{corollary}[theorem]{Corollary}

\theoremstyle{definition}
\newtheorem{definition}[theorem]{Definition}

\newtheorem{remark}[theorem]{Remark}

\numberwithin{equation}{section}

\begin{document}
\title[Tameness of complex dimension in a real analytic set]{Tameness of complex dimension\\ in a real analytic set}

\author{Janusz Adamus}
\address{J. Adamus, Department of Mathematics, The University of Western Ontario, London, Ontario N6A 5B7 Canada
         -- and -- Institute of Mathematics, Faculty of Mathematics and Computer Science,
         Jagiellonian University, ul. {\L}ojasiewicza 6, 30-348 Krak{\'o}w, Poland}
\email{jadamus@uwo.ca}
\author{Serge Randriambololona}
\address{S. Randriambololona, Department of Mathematics, The University of Western Ontario, London, Ontario N6A 5B7 Canada}
\email{serge.randriambololona@ens-lyon.org}
\author{Rasul Shafikov}
\address{R. Shafikov, Department of Mathematics, The University of Western Ontario, London, Ontario N6A 5B7 Canada}
\email{shafikov@uwo.ca}
\thanks{J.Adamus and R.Shafikov were partially supported by Natural Sciences and Engineering Research Council of Canada discovery grants.}

\begin{abstract}
Given a real analytic set $X$ in a complex manifold and a positive  
integer $d$, denote by $\mathcal A^d$ the set of points $p$ in $X$ at which 
there exists a germ of a complex analytic set of dimension $d$ contained in $X$. 
It is proved that $\mathcal A^d$ is a closed semianalytic subset of $X$.
\end{abstract}

\maketitle

\section{Introduction and Main Results}
\label{sec:intro}

Existence or non-existence of complex analytic germs in a given real hypersurface $X$ of a complex manifold
plays an important role in the theory of holomorphic mappings. A particularly interesting case is when $X$ 
is real analytic. For example, in \cite{df} Diederich and Forn\ae ss showed that a compact real analytic set 
$X$ in $\mathbb C^n$ does not contain germs of complex analytic sets of positive dimension. If $X$ is not 
compact, then the set $\mathcal A^1$ of points $p$ in $X$ such that there exists a positive-dimensional complex 
analytic germ $Y_p$ with $Y_p \subset X_p$ is non-empty in general. It is a natural problem to describe the 
structure of the set $\mathcal A^1$. D'Angelo \cite{da2}, and Diederich and Mazzilli \cite{dm} using different 
methods proved that $\mathcal A^1$ is closed in $X$. In \cite{dm} the authors also asked whether 
$\mathcal A^1$ is a real analytic subset of $X$. Our main theorem answers this question.

\begin{theorem}
\label{t:1}
Let $X$ be a closed real analytic subset of an open set in $\mathbb C^n$. Let $\Ad$ denote the set 
of points $p$ in $X$ such that $X_p$, the germ of the set $X$ at $p$, contains a complex analytic germ of 
dimension $d$.  Then $\Ad$ is a closed semianalytic subset of $X$, for every $d\in\N$. Moreover, if $X$ is 
real algebraic, then $\Ad$ is semialgebraic in $X$.
\end{theorem}

The proof of closedness of $\Ad$, given in Proposition~\ref{prop:closed}, is similar in the
spirit to \cite{dm} (where it is done for $\mathcal A^1$), but we do not use volume estimates or Bishop's theorem. 
Instead, our proof purely relies on properties of Segre varieties. The following example, which is due 
to Meylan, Mir, and Zaitsev  \cite{mmz}, shows that the set $\Ad$ is not in general real analytic. 
Consider 
$$
X=\left\{(z_1,\dots, z_4)\in \mathbb C^4 : x_1^2 - x_2^2+x_3^2 = x_4^3\right\},
$$
where $ z_j = x_j+i y_j, \ j=1,\dots,4$. Near $(1,1,0,0)$ the set $X$ is a smooth real algebraic manifold. For
every point $z$ in $X$ with $x_4\ge 0$ there is a complex line passing through $z$ and contained in $X$. But if $x_4<0$, 
then $X$ can be expressed as a graph of a strictly convex function, and therefore there 
cannot be any germs of positive-dimensional complex analytic sets. Thus $\mathcal A^1$ coincides with $X \cap \{x_4\ge 0\}$, 
which is semianalytic (even semialgebraic) but not analytic. 

\begin{remark}
\label{rem:as}
Another (in a sense, dual) question that can be asked about a germ $X_p$ of a real analytic set, is what is the
smallest dimension of a complex analytic germ at $p$ \emph{containing} $X_p$, and what can be said about 
the structure of the subset of $X$ along which this minimal dimension is realized. It is shown in 
\cite[Thm.\,1.5]{as} that for an irreducible real analytic subset $X$ of $\C^n$ of pure dimension $d>0$ this 
so-called \emph{holomorphic closure dimension} attains its minimum $h$ outside a closed semianalytic subset 
$S\subset X$ of dimension less than $d$. In fact, $X\setminus S$ is a CR manifold of CR dimension $d-h$.
Interestingly, $X$ does not in general admit semianalytic (not even subanalytic, see \cite[Ex.\,6.3]{as}) 
stratification by holomorphic closure dimension beyond $S$. (See also \cite{ar} for the semialgebraic context.)
By comparison, Theorem~\ref{t:1} implies a semianalytic filtration of $X$, 
$X=\mathcal A^0\supset\mathcal A^1\supset\dots\supset\mathcal A^{n-1}$.
\end{remark}

Seminanalyticity will be a consequence of the description of the set $\mathcal A^d$ given in Theorem~\ref{t:2} 
below. We first need to introduce some notation. Let $\vr(z,\overline z)$ be a real analytic function on some open
polydisc $V\Subset \mathbb C^n$ given by a power series convergent in a neighbourhood of $\overline V$ such that 
\begin{equation}\label{e:df}
X\cap V = \{z\in V : \vr(z,\overline z)=0\}\,.
\end{equation}
As in the smooth case (see, e.g.,~\cite{w}), for a point $w\in V$, we define the \emph{Segre variety} of $w$ as
\begin{equation}\label{e:sv}
S_w = \{ z\in V : \vr(z,\overline w)=0\}\,.
\end{equation}
For more about Segre varieties see Section~\ref{s:2}. Geometric properties of
these varieties will play a crucial role in the proof of Theorem~\ref{t:2}.

Let $\kp$ be a positive integer, and let $n\geq1$ be the complex dimension of the ambient space of $X$
with variables $z=(z_1,\dots,z_n)$.
For $1\leq d\leq n$, let
\[
\Lambda(d,n):=\{\lb=(\lb_1,\dots,\lb_d)\in\N^d:1\leq\lb_1<\dots<\lb_d\leq n\}\,.
\]
Given $\lb=(\lb_1,\dots,\lb_d)\in\Lambda(d,n)$, we will denote by $z_{\lb}$ the sub-collection of variables $(z_{\lb_1},\dots,z_{\lb_d})$.

\begin{definition}
\label{def:grid}
For any $1\leq d\leq n$, and $\lb\in\Lambda(d,n)$, we define a \emph{$\kp$-grid with $d$-dimensional base $z_{\lb}$}, denoted $\GG^{\kp}_{\lb}$, as follows. Let $\GG^{\kp}_{\lb}$ be a collection of $(\kp+1)^d$ distinct points $p_\nu \in V$, where $\nu=(\nu_1,\dots,\nu_d)\in \{1,\dots,\kp+1\}^d$, such that
\begin{itemize}
\item[(a)] for each pair $(p_{\nu},p_{\nu'})$ of elements of $\GG^{\kp}_{\lb}$, we have $\varrho (p_{\nu},\overline{p_{\nu'}})=0$, and
\item[(b)] for $p_\nu$ and $p_{\nu'}$ in $\GG^{\kp}_{\lb}$, we have $\nu_j = \nu'_j$ if and only if
$p_\nu$ and $p_{\nu'}$ have the same $\lb_j$-th coordinate (as vectors in $\mathbb C^n$).
\end{itemize}
\end{definition}

We denote by  $\mathbb B(p,\ve)$ the standard open Euclidean ball of radius $\ve$ centred at $p$.

\begin{theorem}
\label{t:2}
Let $X$ be a closed real analytic subset of an open set in $\mathbb C^n$, and let $V$ and $\vr$ be such 
that \eqref{e:df} holds. Let $1\leq d< n$, and let $\Ad$ be the set of points $p$ in $X$ such that $X_p$ 
contains a complex analytic germ of dimension $d$. 
Then there exists a positive integer $\kappa$ such that the following two statements are equivalent: 
\begin{enumerate}[{\rm (}i\,{\rm )}]
\item $p\in \Ad\cap V$, 
\item For any $\ve>0$, there exists a $\kp$-grid $\GG^{\kp}_{\lb}$ with a $d$-dimensional base $z_{\lb}$ for some $\lb\in\Lambda(d,n)$ such that 
$\GG^{\kp}_{\lb}\subset \mathbb B(p,\ve)$.
\end{enumerate}
\end{theorem}

In general, the number $\kappa$ in Theorem~\ref{t:2} depends on the defining function $\vr$. However, if $X$ 
is a smooth real analytic hypersurface, then Segre varieties do not depend on the choice of $\vr$ provided that 
the differential of $\vr$ does not vanish on $X$, and in fact, $S_w$ are local biholomorphic invariants of $X$. Thus, in this 
case $\kappa$ is also a local biholomorphic invariant of $X$ (cf. Section~\ref{sec:sub}).
\smallskip

Another question raised in \cite{dm} is whether the set of points on $X$ of infinite D'Angelo type is exactly 
$\mathcal A^1$. The proof of this fact is given in D'Angelo  \cite[Sec.\,3.3.3, Thm.\,4]{da2}, however, in \cite{dm} validity of this proof is questioned. We address this issue in the last section. Our goal is to clarify the definition of 
type for real analytic sets, and to give a concise but self-contained proof of the fact that the subset of $X$ of 
points of infinite type indeed coincides with the set $\mathcal A^1$. Combining this with Theorem~\ref{t:1} immediately gives the following result.

\begin{corollary}
Given a real analytic set $X$, the set of points of D'Angelo infinite type is a closed semianalytic subset of $X$.
\end{corollary}

\section{Segre Varieties}
\label{s:2}

Given a closed real analytic set $X$ in an open set in $\mathbb C^n$ of arbitrary positive dimension, 
for any point $p\in X$ there exists a neighbourhood $V\subset \mathbb C^n$ of $p$ such that $X\cap V$ is precisely 
the zero set of a convergent power series 
\[
\vr(z,\bar{z})=\sum_{|\alpha|+|\beta|\ge1} c_{\alpha\beta}\,(z-p)^{\alpha} {(\overline{z-p})}^{\beta}\,,
\]
where, for a multi-index $\beta=(\beta_1,\dots,\beta_n)\in\N^n$, $w^\beta$ denotes the monomial $w_1^{\beta_1}\dots w_n^{\beta_n}$, and $|\beta|=\beta_1+\dots+\beta_n$. (Indeed, if $X$ is defined near $p$ by the vanishing of real analytic functions $h_1,\dots,h_t$, one can put $\vr=h_1^2+\dots+h_t^2$.)
For simplicity, assume that $p=0$. By shrinking $V$ if needed, we may further assume that the series $\vr(z,\overline w) = \sum c_{\alpha\beta} z^{\alpha} {\overline w}^{\beta}$ is also convergent in a neighbourhood of the closure of $V\times V$. For a given $w\in V$ define the \emph{Segre variety} $S_w$ of $w$ to be the complex analytic subset of $V$ defined by \eqref{e:sv}.

The set
\[
X^c=\{(z,\overline w)\in V\times V: \vr(z,\overline w)=0\}
\]
is a non-empty complex analytic set defined by a single holomorphic function, and hence it is of (pure) dimension $2n-1$. It follows that a fibre 
$\{z\in V:(z,\overline w)\in X^c\}$ over a point $\overline w$, if nonempty, has dimension $n-1$ or $n$. For every point $z\in X$, we have $\vr(z,\overline z)=0$, and 
hence $S_z$ is not empty. Therefore, by the analytic dependence of $S_w$ on $\overline w$, there exist polydisc neighbourhoods $U_1\Subset U_2\Subset V$ of $p$ such 
that for any $w\in U_1$, the set $S_w\cap U_2$ is a non-empty complex analytic subset of $U_2$ of (pure) dimension either $n-1$ or $n$. To simplify notation, 
we will write $S_w$ for $S_w\cap U_2$, whenever $w\in U_1$. From the definition \eqref{e:sv}, and the fact that  $\vr(z,\overline z)$ is real-valued, 
it follows that for $z,w\in U_1$,
\begin{equation}
\label{e:p1}
z\in S_w \Longleftrightarrow w\in S_z,
\end{equation}
\begin{equation}
\label{e:p2}
z\in S_z \Longleftrightarrow z\in X.
\end{equation}
Let $E$ be the set of points $z$ in $U_1$ such that $\dim S_z = n$; i.e., $S_z=U_2$. Then $z\in E$ implies 
$z\in S_z$, and therefore $E\subset X$. Furthermore, $E\ne X$ unless $X$ is itself complex analytic.

\begin{remark}
\label{rem:irreducible}
Apart from properties \eqref{e:p1} and \eqref{e:p2}, the results of the following sections rely on 
a few basic properties of complex analytic sets, 
which we list here for reader's convenience (for details, see \cite{c} or \cite{loj}). 
Let $Y$ denote a complex analytic subset of an open set in $\C^n$.

(1) The family of irreducible components of $Y$ is locally finite, and each irreducible component is precisely the set-theoretic closure in $Y$ 
of a connected component of the regular locus of $Y$.

(2) The set $Y$ is irreducible iff its regular locus $Y^{\reg}$ is a connected manifold. 
In this case, $Y$ is of pure dimension.
Moreover, a proper analytic subset of an irreducible set $Y$ is of dimension at most $\dim Y-1$.

(3) A point $z^0\in Y$ is regular (i.e., $z^0\in Y^{\reg}$) iff there are a natural number $d$, an open 
polydisc $U$ centered at $z_0$, and a sub-collection of variables $(z_{j_1},\dots,z_{j_d})$, such that the projection $\pi$ onto (the linear subspace of $\C^n$ spanned by) these variables
restricted to $Y\cap U$ is a bijection between $Y\cap U$ and $\pi(U)$.

(4) If $Y$ is irreducible, of dimension $k>0$, and $0\in Y$, then
after a (generic) linear change of coordinates in $\C^n$, there is a neighbourhood $\Omega\times\Sigma$ of $0$, 
where $\Omega=\{(z_1,\dots,z_k)\in\C^k:|z_j|<\delta\}$, 
$\Sigma=\{(z_{k+1},\dots,z_n)\in\C^{n-k}:|z_j|<\ve\}$ for some $\delta,\ve>0$, 
and a proper analytic subset $Z$ of $\Omega$, such that the restriction to $Y$, $\pi:Y\cap(\Omega\times\Sigma)\to\Omega$, 
of the canonical projection $\Omega\times\Sigma\to\Omega$ is proper, surjective, 
and locally biholomorphic at every $p$ in $(Y\cap(\Omega\times\Sigma))\setminus(Z\times\Sigma)$, which is an open dense subset 
of $Y\cap(\Omega\times\Sigma)$.

(5) If $\pi$ is a proper projection from $Y$ to a linear subspace of $\C^n$, then $\dim\pi(Y)=\dim Y$.
\end{remark}

By a \emph{holomorphic disc} through a point $p$ we mean an irreducible one-dimensional complex analytic set $Y$ 
in a neighbourhood $U$ of $p$, such that $p\in Y$ and $Y$ is the image of a non-constant holomorphic map $\g$ 
from a disc $\{\zeta\in\C:|\zeta|<\delta\}$ to $U$. We say that the disc is centred at $p$ when $\g(0)=p$.
The following result is essentially a restatement of \cite[Claim on p. 383]{df}. It generalizes \cite[Lem.\,2.5]{dm}, 
which states that a holomorphic disc $Y$ through a point $z$ is contained in $S_z$, provided $Y\subset X$.

\begin{lemma}
\label{lem:dm1}
Let $X,p,\vr,V,U_1$ and $U_2$ be as above. 
Suppose that $Y$ is an irreducible complex analytic subset of an open set in $U_2$, of positive dimension $k$, and such that $Y\subset X$.
Then $z\in Y$ implies $Y\subset S_z$.
\end{lemma}

\begin{proof}
Fix a point $z_0\in Y$. We shall show that $Y\subset S_{z_0}$. For simplicity of notation, assume $z_0=0$. 
By Remark~\ref{rem:irreducible}\,(4), we may choose a neighbourhood $\Omega\times\Sigma$ of $z_0$, 
such that $\Omega$ is a $k$-dimensional polydisc, and the projection $\pi:Y\cap(\Omega\times\Sigma)\to\Omega$ is proper 
and surjective. Let $z'=(z'_1,\dots,z'_k,z'_{k+1},\dots,z'_n)$ be an arbitrary point in $Y\cap(\Omega\times\Sigma)$, 
and let $L_{z'}\subset\Omega$ be the complex line segment through $(z'_1,\dots,z'_k)$ and $0$ in
$\Omega$. Then $Y_{z'}:=\pi^{-1}(L_{z'})$ is an analytic subset of $Y\cap(\Omega\times\Sigma)$, 
with a proper projection onto $L_{z'}$, and hence of dimension one, by Remark~\ref{rem:irreducible}\,(5). 
We may assume that $Y_{z'}$ is irreducible, by keeping only one irreducible component of $Y_{z'}$ passing through $z'$ 
and $z_0$. Then, by the Puiseux theorem (see, e.g., \cite[Ch.\,II, \S\,6.2]{loj}), 
there is a neighbourhood $\Omega'$ of $0\in\Omega$, such that $Y_{z'}\cap(\Omega'\times\Sigma)$ is a holomorphic disc 
centred at $z_0$. By \cite[Lem.\,2.5]{dm}, $Y_{z'}\cap(\Omega'\times\Sigma)\subset S_{z_0}$. 
It follows that the set $Y_{z'}\cap(\Omega'\times\Sigma)\cap S_{z_0}$ contains a non-empty open subset of $Y_{z'}$, 
hence is of dimension $\dim Y_{z'}$, and so is not a proper subset of $Y_{z'}$, by Remark~\ref{rem:irreducible}\,(2). 
Thus $Y_{z'}\subset S_{z_0}$ and, in particular, $z'\in S_{z_0}$. Consequently $Y\cap(\Omega\times\Sigma)\subset S_{z_0}$, because $z'$ was arbitrary. Hence, by Remark~\ref{rem:irreducible}\,(2) again, $Y\subset S_{z_0}$, as required.
\end{proof}

\begin{lemma}[{cf. \cite[Thm.\,1.2]{dm}, see also \cite{df}}]
\label{lem:dm2}
Let $X,p,\vr,V,U_1$ and $U_2$ be as above. For a non-empty subset $Y$ of $U_2$,  with 
$Y\cap U_1 \ne \varnothing$, define
\[
Y^1=\bigcap_{z\in Y\cap U_1}S_z\quad\mathrm{and}\quad Y^2=\bigcap_{w\in Y^1\cap U_1}S_w\,.
\]
Then
\begin{enumerate}
\item $Y^1$ and $Y^2$ are complex analytic subsets of $U_2$. If $Y^1\cap U_1\neq\varnothing$, then $Y\cap U_1\subset Y^2\cap U_1$.
\item Moreover, if $Y$ is an irreducible positive-dimensional complex analytic subset of an open set in $U_2$, such that $Y\subset X$, then $Y\cap U_1\subset Y^1\cap U_1$.
\item If $Y\cap U_1\subset Y^1\cap U_1$, then $Y^2\subset Y^1$ and $Y^2\cap U_1\subset X$.
\end{enumerate}
\end{lemma}

\begin{proof}
(1) The Segre varieties $S_z$ are complex analytic in $U_2$, for $z\in U_1$, hence so are $Y^1$ and $Y^2$.
By definition, $z\in Y^2$ iff $z\in S_w$ for all $w\in Y^1\cap U_1$. Hence, by \eqref{e:p1}, $z\in Y^2\cap U_1$ iff $w\in S_z$ for all $w\in Y^1\cap U_1$. On the other hand, $z\in Y\cap U_1$ implies that $w\in S_z$ for all $w\in Y_1$, and so $z\in Y^2$.

(2) Suppose now that $Y$ is an irreducible positive-dimensional complex analytic subset of an open set in $U_2$, such that $Y\subset X$. Then, by Lemma~\ref{lem:dm1}, $Y\subset S_z$ for every $z\in Y$, and so $Y\cap U_1\,\subset$ $\displaystyle{(\bigcap_{z\in Y\cap U_1}S_z)\cap U_1=Y^1\cap U_1}$.

(3) Finally, assume that $Y\cap U_1\subset Y^1\cap U_1$. Then $\displaystyle{\bigcap_{z\in Y^1\cap U_1}S_z\,\subset\,\bigcap_{z\in Y\cap U_1}S_z}$; i.e., $Y^2\subset Y^1$. For the proof of the last inclusion, let $z\in Y^2\cap U_1$ be arbitrary. Then $z\in S_w$ for every $w\in Y^1\cap U_1$, hence, by \eqref{e:p1} again, $w\in S_z$ for all $w\in Y^1\cap U_1$. In particular, $z\in S_z$, since $z\in Y^2\subset Y^1$. Therefore $z\in X$, by \eqref{e:p2}.
\end{proof}

\section{Topology of the set of points of positive complex dimension}
\label{sec:top}

In this section we prove that $\Ad$ is closed in $X$, for any $d\geq1$. The openness of the set of points of 
finite type in the hypersurface case was established already in \cite[Thm.\,4.11]{da1}, and later extended to 
smooth real analytic sets of arbitrary codimension in \cite{da2}. Via the equivalence between the finiteness of the type at $p$ and the property $p\notin\mathcal A^1$, which we recall in Section~\ref{sec:inf-type}, D'Angelo 
proved in \cite{da2} the openness of $X\setminus\mathcal A^1$. The result was recently reproved in \cite{dm}. 
In the proof of Proposition~\ref{prop:closed} below, we use Lemma~\ref{lem:dm2} to replace complex analytic germs by their representatives in a fixed open set (cf. \cite{dm}), and then show that their Hausdorff limit is contained in a complex analytic set in $X$ that has dimension at least $d$.

For a non-empty set $E\subset\C^n$ and a point $p\in\C^n$, put $d(p,E)=\inf\{d(p,q):q\in E\}$, 
where $d(p,q)$ is the Euclidean distance between $p$ and $q$. Recall that $\overline U_1$ being compact, 
the space $\K(\overline U_1)$ of closed subsets of $\overline U_1$ equipped with the \emph{Hausdorff distance}
\[
d_H(K_1,K_2) = \min \{r\geq0:d(x_1,K_2), d(x_2,K_1)\leq r \ \mathrm{for\ all}\ (x_1,x_2)\in K_1\times K_2 \}
\]
is a compact metric space (see, e.g., \cite{Munkres}). 

\begin{remark}
\label{rem:H-lim}
Suppose that the sequence $(K_j)_{j=1}^{\infty}\subset\K(\overline U_1)$ converges to $K$ in this metric, 
with $d_H(K_j,K)\leq 2^{-j}$. Then $K$ is precisely the set of points $p$ for which there is a sequence 
$(p_j)_{j=1}^{\infty}$ with $p_j \in K_j$ and $d(p_j,p)\leq 2^{-j}$.
In particular, if $K_j \subseteq L_j$ are closed subsets of $\overline U_1$ with 
the sequence $(K_j)$ (resp. $(L_j)$) converging to the set $K$ (resp. $L$), then $K\subseteq L$.
\end{remark}
\medskip

\begin{proposition}
\label{prop:closed}
Let $X$ be a closed real analytic subset of an open set in $\C^n$, and let $\Ad$ be the set of points $p$ in $X$, 
such that $X_p$ contains a complex analytic germ of dimension $d$. Then $\Ad$ is closed in $X$, for every $d\geq1$. 
\end{proposition}

\begin{proof}
Fix $d\geq1$, and let $p_0\in X$ be a limit point of $\Ad$. Then there exists a sequence of $d$-dimensional complex analytic germs $(Y_j)_{p_j} \subset X_{p_j}$  at points $p_j\in X$ such that $p_0 = \lim_{j\to\infty} p_j$. 
We restrict our considerations to neighbourhoods $U_1$ and $U_2$ of $p_0$, as discussed in Section~\ref{s:2}.
Without loss of generality, we may assume that the $Y_j$ are irreducible.

One difficulty arising here is that the $(Y_j)_{p_j}$ may not simultanously admit representatives in a fixed 
neighbourhood of $p_0$. We can, however, replace the $Y_j$ by irreducible complex analytic subsets of $U_2$ by 
setting
\[
Y^1_j\,= \bigcap_{z\in Y_j\cap U_1} S_z \qquad\mathrm{and}\qquad Y^2_j\,= \bigcap_{w\in Y^1_j\cap U_1} S_w\,.
\]
Indeed, by Lemma~\ref{lem:dm2}, the $Y_j^1$ and $Y_j^2$ are complex analytic subsets of $U_2$, 
$Y_j\subset Y^2_j$ and $Y^2_j\cap U_1\subset X$. 
The first inclusion implies also that $\dim Y^2_j\geq d$, for all $j$, since the $Y_j$ are $d$-dimensional. 
We may also assume that the $Y_j^2$ are irreducible, by 
keeping only one irreducible component of $Y^2_j$ passing through $p_j$.
To simplify the notation, from now on we denote $Y^2_j$ by $Y_j$. Since $\dim Y_j\in\{d,\dots,n-1\}$ for all $j$, 
there exists an integer $d'\geq d$ such that $\dim Y_j=d'$ for infinitely many $j$. 
Let us then replace the original sequence $(Y_j)_{j=1}^{\infty}$ by this infinite subsequence.

By compactness of $\K(\overline U_1)$, the sequence $(Y_j\cap \overline U_1)_{j=1}^{\infty}$ contains an 
infinite subsequence convergent in the Hausdorff metric to a set $Y_0$ closed in 
$\overline U_1$. Therefore, without loss of generality, we may assume that
\[
Y_0={\lim}_H(Y_j\cap\overline U_1)\,,
\]
and further that 
\begin{equation}
\label{eq:fast}
d_H(Y_j\cap\overline U_1,Y_0)\leq2^{-j}
\end{equation}
 (by throwing out some terms of the sequence, if necessary).
Notice that $p_0=\lim_{j\to\infty}p_j$ belongs to $Y_0$, by Remark~\ref{rem:H-lim}.

We will show that $Y_0$ is contained in a complex analytic set, of dimension at least $d$, contained in $X$.
Set
\[
Y^1_j = \bigcap_{z\in Y_j\cap U_1} S_z, \quad Y^1_0 = \bigcap_{z\in Y_0\cap U_1} S_z, \quad{\rm and\ \ } 
\tilde Y^1_0 = {\lim}_H(Y^1_j\cap\overline U_1)\,,
\]
where ${\lim}_H(Y^1_j\cap\overline U_1)$ is again the limit of (an infinite convergent subsequence of) 
$Y^1_j\cap \overline U_1$ in the sense of the Hausdorff metric on $\K(\overline U_1)$. (Notice that replacing 
$(Y_j\cap \overline U_1)_{j=1}^{\infty}$ by its infinite convergent subsequence does not affect $Y_0$.)
 We may further assume that $d_H(Y^1_j\cap\overline U_1,\tilde Y^1_0)\leq2^{-j}$, as above.

We claim that $\tilde Y^1_0 \subset Y^1_0$. Indeed, there exist points $\{a_1,\dots, a_r\}\subset Y_0$ such
that $Y^1_0=\bigcap_{k=1}^r S_{a_k}$, by compactness of $\overline U_2$ and Remark~\ref{rem:irreducible}\,(1).
Therefore, there exist $r$ sequences $(a_k^j)_{j=1}^\infty$, such that $a_k^j\in Y_j$ and $\lim_{j\to\infty}a_k^j=a_k$, $k=1,\dots, r$ (see Remark~\ref{rem:H-lim}). 
From the analytic dependence of Segre varieties $S_z$ on the parameter $z$, we conclude that
\[
{\lim}_H (\bigcap_{k=1}^r S_{a^j_k}) \subset \bigcap_{k=1}^r S_{a_k}=Y^1_0\,;
\]
for if $z\in{\lim}_H \bigcap_{k=1}^r S_{a^j_k}$, we can find $z^j\in \bigcap_{k=1}^r S_{a^j_k}$ such that $\lim_{j\to\infty}z^j=z$, hence
\[
\vr(z,\overline a_k)=\lim_j \vr(z^j, \overline a^j_k)=0
\]
for each $k\in \{1 ,\cdots , r\}$.

Also, since $a^j_k\in Y_j$, for every fixed $j$ we have $Y^1_j \subset \bigcap_{k=1}^r S_{a^j_k}$.
From this we conclude that ${\lim}_H Y^1_j \subset Y^1_0$, which proves the claim.
\medskip

We now claim that $Y_0\cap U_1 \subset Y^1_0\cap U_1$. Indeed, since the $Y_j\cap U_1$ are irreducible positive-dimensional complex analytic sets in $U_1$, and 
subsets of $X$, we have $Y_j\cap U_1\subset Y^1_j\cap U_1$, by Lemma~\ref{lem:dm2}(2). Therefore, by Remark~\ref{rem:H-lim},
${\lim}_H (Y_j\cap U_1) \subset {\lim}_H (Y^1_j\cap U_1)= \tilde Y^1_0$, and hence $Y_0\cap U_1={\lim}_H (Y_j\cap U_1) \subset Y^1_0\cap U_1$, by the previous claim. 
In particular, the set $Y^1_0\cap U_1$ is not empty. Let 
\[
Y^2_0 = \bigcap_{z\in Y^1_0\cap U_1} S_z\,.
\]
Then $Y^2_0\subset U_2$ is a complex analytic set, such that $Y^2_0\cap U_1\subset X$ and $\dim_{p_0} Y^2_0\geq d$. 
Indeed, since $Y_0\cap U_1 \subset Y^1_0\cap U_1$, Lemma~\ref{lem:dm2} implies that $Y^2_0\cap U_1\subset X$.
Given $z\in Y_0\cap U_1$, we have $w\in S_z$ for every $w\in Y^1_0$, by definition of $Y^1_0$. Hence $z\in S_w$ for every $w\in Y^1_0\cap U_1$, by \eqref{e:p1}, and so 
$z\in Y^2_0$. Therefore $Y_0\cap U_1\subset Y_0^2$. It thus suffices to show that the Hausdorff dimension of $(Y_0)_{p_0}$ is at least $2d'$.
This is a consequence of \cite[Thm.\,4.2]{vdd}, but one can also argue directly as follows.

Recall that, for every $j\geq1$, $Y_j$ is an irreducible $d'$-dimensional complex analytic subset of $U_2$ (where $d'\geq d$) passing through $p_j$, and such that 
$Y_j\cap U_1\subset X$. By (\ref{eq:fast}), we have
\begin{equation}
\label{eq:estimate}
d_H(Y_j\cap\overline U_1,Y_{j+k}\cap\overline U_1)<2^{-(j-1)}\,.
\end{equation}
Since $\lim_{j\to\infty}p_j=p_0$, it follows that, for every $\delta=(\delta_1,\dots,\delta_n)$ with $\delta_l>0$, all but finitely many $Y_j$ have non-empty intersection 
with a polydisc $P(p_0,\delta)=\{z=(z_1,\dots,z_n)\in\C^n:|z_l-p_{0l}|<\delta_l\}$. For every $j$, there exist 
$\delta$ and a generic system of coordinates 
$z=(z_1,\dots,z_{d'},z_{d'+1},\dots,z_n)$ at $p_0$, such that $Y_j\cap P(p_0,\delta)$ has proper and surjective projection onto the $(z_1,\dots,z_{d'})$-variables 
(see Remark~\ref{rem:irreducible}\,(4)). By \eqref{eq:estimate}, we may choose a positive $\delta$ and a system of coordinates $z$ at $p_0$ 
such that all but finitely many of the $Y_j\cap P(p_0,\delta)$ simultaneously have proper and surjective projection onto the $(z_1,\dots,z_{d'})$-variables. 
Therefore the same must be true for the Hausdorff limit $Y_0\cap P(p_0,\delta)$, by Remark~\ref{rem:H-lim}. 
Thus the Hausdorff dimension of $(Y_0)_{p_0}$ is at least $2d'\geq2d$, and hence $p_0\in\Ad$, which completes the proof of the proposition.
\end{proof}

\section{Finiteness and Noetherianity in analytic families}
\label{sec:sub}

In this section we prove two finiteness properties for intersections of 
elements in a family of analytic sets that will be used in the proof of 
Theorem~\ref{t:2}. We begin with some basic facts about semi- and subanalytic sets.

Recall that a subset $E$ of a real analytic manifold $M$ is called {\em semianalytic} if it
is locally defined by finitely many real analytic equations and inequalities. More precisely,
for each $p\in M$, there is a neighbourhood $U$ of $p$, and real 
analytic in $U$ functions $f_i,g_{ij}$, where $i=1,\ldots, r$, $j=1,\ldots, s$, such that
\[
E\cap U =\bigcup_{i=1}^r\left(\bigcap_{j=1}^s\{x\in U\colon g_{ij}(x)>0 \text{ and } f_i(x)=0\}\right)\,.
\]
A real analytic set is clearly semianalytic. A \emph{subanalytic} subset $E$ of a real analytic manifold $M$ is one which can be locally represented as the proper projection of a semianalytic set. More precisely, for every $p\in M$, there exist a neighbourhood
$U$ of $p$ in $M$, a real analytic manifold $N$, and a relatively 
compact semianalytic set
$Z\subset M\times N$ such that $E\cap U=\pi(Z)$, where $\pi:M\times
N\rightarrow M$ is the natural projection. In particular, semianalytic 
sets are subanalytic.
For details on semi- and subanalytic sets we refer the reader to \cite{bm}.

The class of semianalytic (resp. subanalytic) sets is closed under natural topological operations:
locally finite unions and intersections, set-theoretic differences, complements,
topological closures and interiors of semianalytic (resp. subanalytic) sets are semianalytic (resp. subanalytic). 
Subanalytic sets are furthermore closed under the operation of taking proper projections to linear subspaces.

\begin{remark}
\label{rem:fibre-comps}
An important property of subanalytic sets is that the number of connected components of fibres of a projection is locally bounded 
(see, e.g., \cite[Thm.\,3.14]{bm}): If $S$ is a relatively compact subanalytic subset of $\R^m \times \R^n$, and $D\subset\R^m$ is compact, 
then there is a positive integer $k_D$ such that the number of connected components of the set $\pi^{-1}(x)$ is bounded above by $k_D$ for all 
$x\in D$, where $\pi$ is the restriction to $S$ of the canonical projection $\R^m\times\R^n\to\R^m$.
\end{remark} 

\begin{lemma}
\label{lem:fin1}
Let $S$ be a subanalytic subset of $\C^m\times\C^n$. Let $\Omega_1$ and $\Omega_2$ be relatively compact open subsets of $\C^m$ and $\C^n$ respectively, 
and let $D_1\subset\C^m$ and $D_2\subset\C^n$ be open polydiscs, such that $\overline D_1\subset\Omega_1$ and $\overline D_2\subset\Omega_2$.
Suppose that for every point $a\in D_1$, the set $S_a=\{ b\in\Omega_2\,:\,(a,b)\in S \}$ is a complex analytic subset of $\Omega_2$. 
Then there is a positive integer $N$ such that, for every $a\in D_1$, the analytic set $S_a\cap D_2$ has at most $N$ irreducible components.   
\end{lemma}

\begin{proof}
By Remark~\ref{rem:irreducible}\,(1), it suffices to show that there is a positive integer $N$ such that for every $a\in D_1$, the analytic set 
$(S_a\cap D_2)^{\reg}$ has at most $N$ connected components. Using Remark~\ref{rem:fibre-comps}, the latter would be a consequence of the 
subanalyticity of the set
\[
\{ (a,b)\in D_1\times D_2 :  b\in (S_a\cap D_2)^{\reg}\}\,.
\]
Remark~\ref{rem:irreducible}\,(3) ensures that this set is precisely the set of pairs $(a,b)$ in $D_1\times D_2$ 
for which there is a natural number $d$ and a choice of coordinate indices $(j_1,\dots j_d)\in \{1,\dots,n\}^d$, such that 
there is a number $\ve >0$ small enough so that for all $(z_{j_1},\dots,z_{j_d}) \in \C^d$ with $|z_{j_l}-b_{j_l}|<\ve$  
($l=1,\dots,d$) there is a unique $b'=(b'_1,\dots,b'_n)$ satisfying $b'\in Y\cap \mathbb B(b,\ve)$ and $b'_{j_l}=z_{j_l}$ ($l=1,\dots,d$).

The set  $\{ (a,b)\in D_1\times D_2 :  b\in (S_a\cap D_2)^{\reg}\}$ is thus the proper projection (\emph{``there exists''}) of the complement of the proper projection (\emph{``for all''}) of the complement of the proper projection of a semianalytic set, and is therefore subanalytic.
\end{proof}

Using this lemma, we can now prove the following proposition.

\begin{proposition}
\label{prop:fin2}
Under the notation of the previous lemma, there is a positive integer $L$ such that for any set $A \subset D_1$ 
there is an $L$-tuple $( a_1,\cdots, a_L ) \in A^L$ for which
\[
(\bigcap _{a\in A} S_a )\cap D_2= S_{a_1} \cap \cdots \cap S_{a_L} \cap D_2\,.
\]
\end{proposition}

\begin{proof}
Given $l\geq1$ and $(a_1,\cdots,a_l)\in (D_1)^l$, let $N(l;a_1, \cdots, a_l)$ denote the $(n+1)$-tuple of natural numbers whose $k$'th coordinate 
is the number of irreducible components of dimension $n-k+1$ of $S_{a_1}\cap \ldots S_{a_l}\cap D_2$. 

Applying Lemma~\ref{lem:fin1} to the subanalytic set 
$\{(a_1,\ldots,a_l,b)\in (\C^m)^l\times \C ^n : b\in S_{a_1} \cap \ldots \cap S_{a_l}\}$, 
we conclude that the number of such components of any dimension is bounded above independently 
of the choice of $(a_1,\cdots, a_l)$ 
(but {\it a priori} not independently of $l$). 
Hence $N(l;a_1, \cdots, a_l)$ is well-defined for all $(a_1,\dots,a_l)\in(D_1)^l$, and the set
\[
\{ N(l;a_1, \cdots, a_l) : (a_1,\cdots,a_l)\in(D_1)^l \}
\]
is a finite subset of $\N^{n+1}$. 

Let us order $\N^{n+1}$ lexicographically. Observe that 
\begin{equation}
\label{decroit}
N(l;a_1, \cdots, a_l)\geq _{lex} N(l+1;a_1, \cdots, a_{l+1})
\end{equation}
for any $(a_1,\cdots , a_{l+1})\in(D_1)^{l+1}$. Indeed, by intersecting $S_{a_1} \cap \cdots \cap S_{ a_l}$ with $S_{a_{l+1}}$ 
we may only decrease lexicographically the number of irreducible components: an irreducible component $Z_{\mu}$ of $S_{a_{l+1}}$ 
either contains all the irreducible components of $S_{a_1} \cap \cdots \cap S_{ a_l}$, in which case our $(n+1)$-tuple is not affected, 
or else there is an irreducible component $W_{\nu}$ of $S_{a_1} \cap \cdots \cap S_{ a_l}$, of dimension, say,  $k$, such that 
$Z_{\mu}\cap W_{\nu}\varsubsetneq W_{\nu}$. In the latter case, by Remark~\ref{rem:irreducible}\,(2), the set $Z_{\mu}\cap W_{\nu}$ is of 
dimension strictly smaller than $k$, and so the number of $k$-dimensional components in $S_{a_1} \cap \cdots \cap S_{ a_{l+1}}$ is strictly 
less that that in $S_{a_1} \cap \cdots \cap S_{ a_l}$.

Suppose for a contradiction that the number $L$ from the proposition does not exist. Then for every $l\geq1$, the set
\begin{multline*}
T_l\ :=\ \{\,(a_1,\dots,a_l)\in D_1^l :\\ N(1;a_1)>_{lex} N(2;a_1,a_2) >_{lex}\dots >_{lex} N(l;a_1, \dots, a_l) \,\}
\end{multline*}
is nonempty. Let $N(l)$ be the (lexicographic) maximum among the tuples $N(l;a_1,\dots,a_l)$ as $(a_1,\dots,a_l)\in T_l$, and let 
$(b_1^l , \cdots ,b_l^l)\in T_l$ be such that $N(l)=N(l;b_1^l,\dots,b_l^l)$.
It follows that
\[
N(l)\;\geq _{lex} N(l;b^{l+1}_1, \cdots, b^{l+1}_l)\;>_{lex} N(l+1;b^{l+1}_1, \cdots, b^{l+1}_{l+1})\;=\; N(l+1)\,,
\]
for all $l\geq1$. Hence there exists a strictly decreasing infinite sequence of $(n+1)$-tuples
\[
N(1)\ >_{lex}N(2)\ >_{lex}\dots\ >_{lex}N(l)\ >_{lex}\dots\,,
\]
which contradicts the fact that $\geq _{lex}$ is a well-ordering of $\N^{n+1}$.
\end{proof}

For $1\leq d < n$, and $\lb\in \Lambda(d,n)$, let 
\begin{equation}\label{e:proj}
\pi_{\lb}=\pi_{\lb_1,\dots,\lb_d}: \mathbb C^n\to \mathbb C^d
\end{equation}
be the canonical projection
from $\C^n$ onto (its linear subspace spanned by) the variables $z_{\lb}=(z_{\lb_1},\dots,z_{\lb_d})$. Let
$z_{\mu}=(z_{\mu_1},\cdots,z_{\mu_{n-d}})$ be the $(n-d)$-tuple of the remaining variables (that is, 
$\{1,\cdots, n\}=\{\lb_1,\dots,\lb_d\}\cup\{\mu_1, \cdots , \mu_{n-d} \}$, with 
$1\leq \mu_1< \cdots < \mu_{n-d}\leq n$).

\begin{corollary}
\label{rem:fin3}
Under the notation of Lemma~\ref{lem:fin1}, there exists a positive integer $\kp$ such that for every non-empty $A\subset D_1$ and any $\lb$, the number of irreducible components of a fibre of $\pi_{\lb}|_{(\bigcap _{a\in A} S_a \cap D_2)}$ is bounded above by $\kp$.
\end{corollary}

\begin{proof}
Use Proposition~\ref{prop:fin2} to replace $\bigcap _{a\in A} S_a \cap D_2$ by some $S_{a_1}\cap \cdots \cap S_{a_L}$ and then apply Lemma~\ref{lem:fin1} to the sets
\[
\{(a_1,\cdots, a_L, z_{\lb},z_{\mu})\in (D_1^L \times \C ^d) \times \C^{n-d} : z\in S_{a_1}\cap \cdots \cap S_{a_L}\cap D_2 \}.
\]
\end{proof}

\section{Proofs of the main theorems}
\label{sec:main-proofs}
 
We first prove Theorem~\ref{t:2}, from which the semianaliticity in Theorem~\ref{t:1} will follow.

\subsection{Proof of Theorem~\ref{t:2}}

Fix $d\geq1$. We give the proof of Theorem~\ref{t:2} for this given dimension.

\smallskip

({\it i\,})\,$\Longrightarrow$({\it ii\,}). Let $p\in\Ad$ be an arbitrary point, and let $U_1$ and $U_2$ be neighbourhoods
of $p$ as defined in Section~\ref{s:2}. Then there exists a complex analytic set $Y$ in a 
neighbourhood of $p$, of dimension $d$, which is contained in $X$ and passes through $p$. We may assume 
that $Y$ is irreducible, and hence, by Lemma~\ref{lem:dm2}\,(2), 
$\displaystyle{\bigcap_{z\in Y\cap U_1} S_z}$ contains $Y\cap U_1$. 

Let $(z_1,\dots,z_n)$ be the coordinates in $\C^n$. We will show that, for every $\ve>0$ and $\kappa>0$, 
there exists $\lb\in\Lambda(d,n)$ for which there is a $\kp$-grid $\GG^{\kp}_{\lb}$ with 
$d$-dimensional base $z_{\lb}$ such 
that $\GG^{\kp}_{\lb}\subset \mathbb B(p,\ve)$.

Fix $\ve>0$. By Remark~\ref{rem:irreducible}\,(4) there are a small polydisc 
$D\subseteq \mathbb B (p,\ve)\cap U_1$ such that 
$Y\cap D$ is a complex manifold, and $\lb=(\lb_1,\dots,\lb_d)\in\Lambda(d,n)$ such that $Y\cap D$ is the 
graph of a holomorphic mapping in variables $z_{\lb}$. In particular, any set 
\begin{multline*}
\{z_\nu \in \pi_{\lambda} (D), \nu = (\nu _1, \ldots , \nu _d)\in \{1,\ldots, \kp+1\}^d :\\ 
{\rm \ for \  all\ } \nu,\nu',j, \  
\nu_j = \nu'_j \Leftrightarrow \pi_{\lb_j}(z_{\nu})=\pi_{\lb_j}(z_{\nu'}) \}
\end{multline*}
is pulled back by $\pi_{\lb}|_{Y\cap D}$ to a set 
$$
\GG^{\kp}_{\lb}=\left\{p_\nu: \nu=(\nu_1,\dots,\nu_d)\in\{1,\dots,\kp+1\}^d\right\}
$$ 
satisfying (b) of Definition~\ref{def:grid} ($\pi_\lb$ and $\pi_{\lb_j}$ are as in~\eqref{e:proj}). But as 
noted earlier, $$\bigcap_{z\in Y\cap U_1} S_z \supset Y\cap U_1,$$ which shows that  $\GG^{\kp}_{\lb}$ 
also satisfies (a) of Definition~\ref{def:grid}.

\medskip

({\it ii\,})\,$\Longrightarrow$({\it i\,}). Let $q\in X\cap V$ be arbitrary, and let again, 
$U_1$ and $U_2$ be neighbourhoods
of $q$ as defined in Section~\ref{s:2}. Let $\kp\geq 1$ be an upper bound for the number of irreducible 
components of any fibre of $\bigcap _{z\in Z} S_z \cap U_2$ for any projection $\pi_{\lb}$, $\lb \in \Lambda(d,n)$, 
as $Z$ ranges over the subsets of $U_1$. Corollary \ref{rem:fin3} applied to the set
$\{(a,\overline b)\in \C^n \times \C ^n : \vr(a, \overline b)=0 \}$ insures that this upper bound is finite.

Let $p\in U_1\cap X$. Suppose that for any $\ve>0$, there exists a $\kp$-grid with $d$-dimensional base $z_{\lb}$ 
for some $\lb=(\lb_1,\dots,\lb_d)$,
\[
\GG^{\kp}_{\lb}=\left\{p_\nu: \nu=(\nu_1,\dots,\nu_d)\in\{1,\dots,\kp+1\}^d\right\}
\]
contained in $\mathbb B (p,\ve)$. Without loss of generality we may assume that the open 
$\ve$-ball $\mathbb B(p,\ve)$ is contained in $U_1$. 

Let $Y^1=\bigcap_{z\in\GG^{\kp}_{\lb}} S_z$ and $Y^2=\bigcap_{z\in Y^1 \cap U_1} S_z$. 
By Lemma~\ref{lem:dm2}\,(1), $\GG^{\kp}_{\lb}\subset Y^2$; 
moreover $Y^2 \subset X$, by Definition~\ref{def:grid}\,(a) and Lemma~\ref{lem:dm2}\,(3).

For $\lb$ as above, we denote by $\lb ^{(\delta)}$ the $\delta$-tuple $(\lb_1,\cdots,\lb_{\delta})\in 
\Lambda (\delta,n)$ of the first $\delta$ components of~$\lb$, $\delta\in \{1,\ldots, d \}$.
We will consider the fibres $\pi^{-1}_{\lb^{(\delta)}} \left(\pi_{\lb^{(\delta)}} (p_{\nu})\right)$ at points $p_\nu\in\GG^{\kp}_{\lb}$, with the convention that $\pi^{-1}_{\lb^{(0)}}(\pi_{\lb^{(0)}}(p_{\nu}))=V$.

Let us prove by descending induction on $\delta\in \{0, \ldots, d\}$ that for each 
$p_{\nu}\in \GG^{\kp}_{\lb}$ the fibre 
$$
\pi^{-1}_{\lb^{(\delta)}} \left(\pi_{\lb^{(\delta)}} (p_{\nu})\right)\cap Y^2
$$ 
contains an irreducible component of dimension $\geq d-\delta$ that passes through a $p_{\nu'} \in \GG^{\kp}_{\lb}$ with 
$\pi_{\lb^{(\delta)}}(p_{\nu})=\pi_{\lb^{(\delta)}}(p_{\nu'})$ (the latter equality being vacuously true if 
$\delta=0$).

$\bullet$ For $\delta=d$, it suffices to take any irreducible component of 
$\pi^{-1}_{\lb} (\pi_{\lb} (p_{\nu}))\cap Y^2$ passing through $p_{\nu}$ (which exists 
since  $p_{\nu}\in Y^2$).

$\bullet$ Suppose the result holds for $\delta +1$. Then the collection of subsets of $V$
\[
\left\{ \pi^{-1}_{\lb^{(\delta+1)}} (\pi_{\lb^{(\delta +1)}} (p_{\mu}))\cap Y^2: 
p_{\mu}\in \GG^{\kp}_{\lb} , \ \pi_{\lb^{(\delta)}} (p_{\mu})=
\pi_{\lb^{(\delta)}} (p_{\nu}) \right\}
\]
has $\kappa+1$ pairwise disjoints elements (one for each $\pi_{\lb^{(\delta +1)}} (p_{\mu})$), each 
containing an irreducible component of dimension $\geq d-(\delta+1)$ and each contained in 
$\pi^{-1}_{\lb^{(\delta)}} (\pi_{\lb^{(\delta)}} (p_{\nu}))\cap Y^2 $. By definition of $\kappa$ and the pigeonhole principle, there is an irreducible component 
$X_{\nu}$ of $\pi^{-1}_{\lb^{(\delta)}} (\pi_{\lb^{(\delta)}} (p_{\nu}))\cap Y^2 $ and two indices 
$\mu$ and $\mu '$ such that 
$\pi_{\lb^{(\delta +1)}} (p_{\mu}) \neq \pi_{\lb^{(\delta +1)}} (p_{\mu'})$, and there is an irreducible component $X_{\mu}$  (resp. $X_{\mu'}$) of  
$ \pi^{-1}_{\lb^{(\delta+1)}} (\pi_{\lb^{(\delta +1)}} (p_{\mu}))$ 
(resp. of  $ \pi^{-1}_{\lb^{(\delta+1)}} (\pi_{\lb^{(\delta +1)}} (p_{\mu'}))$) of dimension 
$\geq d-(\delta+1)$ with
\[
X_{\mu} \subset X_{\nu} \mbox{ and } X_{\mu'} \subset X_{\nu}.
\]

Since $X_{\mu}\cap X_{\mu'}=\varnothing$, we get $\dim X_{\nu}\geq d-\delta$, for else $X_{\nu}$ would be the union of proper analytic subsets 
$X_{\mu}$, $ X_{\mu'}$ and $\overline{X_{\nu}\setminus(X_{\mu}\cup X_{\mu'})}$, with $\dim X_{\mu}=\dim X_{\mu'}=\dim X_{\nu}$, contradicting irreducibility of $X_{\nu}$ 
(Remark~\ref{rem:irreducible}\,(2)).

The case $\delta=0$ of the induction provides a point $p_{\nu '} \in \mathbb B(p,\ve) \cap \Ad$. Therefore, 
$p$ is an accumulation point of $\Ad$, and hence $p\in\Ad$, by Proposition~\ref{prop:closed}.

\smallskip

Finally, for any point $q\in V$, there is a pair of neighbourhoods 
$U^q_1\Subset U^q_2\Subset V$ such that for every $w\in U^q_1$, $S_w$ is a complex analytic subset 
of $U^q_2$ of dimension at least $n-1$ (cf. Section~\ref{s:2}). Since $V$ is relatively compact in the domain of 
convergence of $\vr$, the set $X\cap V$ can be covered by a finite collection of open sets 
$U^{q_\alpha}_1$, $\alpha=1,\dots, N$.
Taking the maximum value among the $\kappa$ associated to each $U^{q_\alpha}_2$ 
will give the uniform $\kappa$, as claimed in Theorem~\ref{t:2}.

\qed

\subsection{Proof of Theorem~\ref{t:1}}

Theorem~\ref{t:2} gives us a description of $\Ad$, $d\geq1$, as a subanalytic set. This
description will be shown to actually define a semianalytic set which will prove Theorem \ref{t:1}.

Let $p\in X$ be arbitrary. Let $\vr(z,\overline z)$ be any defining function of $X$ given by a
convergent power series in a polydisc neighbourhood $V$ of $p$. Let $\kp$ be as in Theorem~\ref{t:2}. 
Define
\[
\Sigma_1=\left\{(z_1,\dots, z_{\kp+1})\in V^{\kp+1} :\ \vr(z_\mu, \overline z_\nu)=0,\ \ 1 \le \mu,\nu \le \kp+1 \right\}\,.
\]
Then $\Sigma_1$ is a real analytic subset of $V^{\kp+1}$. Let 
$\Delta_1=\{(z_1,\dots,z_{\kp+1})\in(\C^{n})^{\kp+1}: z_1=\dots=z_{\kp+1}\}$, and consider the set
\[
S_1=\overline{\Sigma_1 \setminus \{(z_1,\dots,z_{\kp+1})\in V^{\kp+1} :\ 
z_\nu = z_{\nu'}\ {\rm \ for\ some\ }\nu \ne \nu'\}}\;\cap\; \Delta_1\,.
\]
The closure of a semianalytic set being semianalytic, $S_1$ is a semianalytic subset of the diagonal  $\Delta_1$. 
One easily checks that the projection to the first coordinate of a semianalytic subset of the diagonal is itself 
semianalytic. But $\mathcal A^1\cap V$ is precisely the projection of $S_1$ to the first coordinate, by Theorem~\ref{t:2}.
\smallskip

Similarly, for $d\geq2$, define
\begin{multline*}
\Sigma_d=\left\{(z_{1,\dots,1},\dots,z_{\kp+1,\dots,\kp+1})\in V^{(\kp+1)^d}: \right.\\
\left. \vr(z_\nu,\overline z_{\nu'})=0,\ \ \nu,\nu'\in \{1 ,\ldots ,\kappa+1 \}^d \right\} \, ,
\end{multline*}
and for every $\lb=(\lb_1,\dots,\lb_d)\in\Lambda(d,n)$, put
\begin{multline*}
\Theta^d_{\lb}=\left\{(z_{1,\dots,1},\dots,z_{\kp+1,\dots,\kp+1})\in V^{(\kp+1)^d}:\right. \ \mathrm{for\ all\ }j\in \{1,\dots,d\} \\ 
\left.\mbox{ and }  (\nu , \nu') \in (\{ 1,\ldots \kappa +1\}^d)^2\ , 
 \pi_{\lb_j}(z_{\nu})=\pi_{\lb_j}(z_{\nu'}) \Leftrightarrow \nu_j = \nu'_j \right\}.
\end{multline*}
Then $\Sigma_d\,\cap\,\bigcup_{\lb\in\Lambda(d,n)}\Theta^d_{\lb}$ is a semianalytic subset of 
$V^{(\kp+1)^d}$.  
Let 
\[
\Delta_d=\{(z_{1,\dots,1},\dots,z_{\kp+1,\dots,\kp+1})\in V^{(\kp+1)^d}: 
z_{1,\dots,1}=\dots=z_{\kp+1,\dots,\kp+1}\}
\]
and consider the set
\begin{multline*}
S_d=\left. \overline{(\Sigma_d\,\cap\,\bigcup_{\lb\in\Lambda(d,n)}\Theta^d_{\lb}) \setminus 
\{(z_{1,\dots,1},\dots,z_{\kp+1,\dots,\kp+1})\in V^{(\kp+1)^d}:\ } \right.
\\ \left. \overline{ \ z_\nu = z_{\nu'}\ {\rm \ for\ some\ }\nu \ne \nu'\}}\;\cap\; \Delta_d\right. \,. 
\end{multline*}
As above, $S_d$ is a semianalytic subset of the diagonal $\Delta_d$, and hence its projection to the first coordinate, 
which is precisely $\Ad\cap V$ (by Theorem~\ref{t:2}), is itself semianalytic.

Finally, suppose that $X$ is real algebraic. Then $\vr$ is a polynomial, and hence the sets $\Sigma_d$ 
above are all semialgebraic. 
It follows that the $\Ad$ are semialgebraic, for all $d\in\N$, which completes the proof of Theorem~\ref{t:1}.
\qed

\section{Appendix: Points of infinite type}
\label{sec:inf-type}

In this section we review the basics of D'Angelo's theory of points of finite type. Let, as before, $X$ denote a closed real analytic subset of an open set in $\C^n$. Our goal is to clarify the definition of type in the case that $X$ is not a smooth hypersurface, and to give a condensed but self-contained proof of the fact that the subset of $X$ of points of infinite type coincides with $\mathcal A^1$ (cf. \cite[\S\,3.3.3, Thm.\,4]{da2}). We were motivated, in part, by the claims of incompleteness of the D'Angelo argument (see \cite{dm}). All the proofs presented in this section (modulo minor technical modifications) originate in D'Angelo \cite{da1} and \cite{da2}.

\subsection{Order of contact of a holomorphic ideal}

Let $\OO_p={_n\OO_p}$ denote the ring of germs of holomorphic functions at a point $p=(p_1,\dots,p_n)\in\C^n$. 
By the Taylor expansion isomorphism, we may identify $_n\OO_p$ with the ring $\C\{z-p\}$ of convergent power series in $z-p$, where $z=(z_1,\dots,z_n)$ is a system of $n$ complex variables. 
Let $\mm_p$ denote the maximal ideal of the local ring $_n\OO_p$. Let $\Hp$ denote the set of germs of 
(non-constant) holomorphic mappings from a neighbourhood of $0$ in $\C$ to a neighbourhood of $p$ in $\C^n$ 
(sending $0$ to $p$).
Given $f=(f_1,\dots,f_n)\in\C\{\zeta\}^n$, we denote by $\nu(f)$ the order of vanishing of $f$ at $0$; i.e., $\nu(f):=\max\{k\in\N:f_j\in\mm^k, j=1,\dots,n\}$ if $f\neq0$ in $\C\{\zeta\}^n$, and $\nu(0):=\infty$, where $\mm$ is the maximal ideal of $\C\{\zeta\}$.

\begin{definition}[{\cite[Def.\,2.6]{da1}}]
\label{def:orders}
Given a proper ideal $I$ in $\OO_p$, define
\begin{align*}
\tau^*(I)&=\sup_{\gamma\in\Hp}\ \inf_{g\in I}\ \frac{\nu(g\circ\gamma)}{\nu(\gamma)}\,;\\
K(I)&=\inf\{k\in\N:\mm_p^k\subset I\}\,;\\
D(I)&=\dim_{\C}(\OO_p/I)\qquad\mathrm{(as\ a\ complex\ vector\ space)}.
\end{align*}
\end{definition}

The following is a simplified variant of \cite[Thm.\,2.7]{da1}.

\begin{lemma}
\label{lem:ineq1}
Suppose that $I$ is a proper ideal in $\OO_p$. Then
\[
\tau^*(I)\ \leq\ K(I)\ \leq\ D(I)\,.
\]
Moreover, each of the above constants is finite iff the zero-set germ of $I$ is the singleton $\{p\}$.
\end{lemma}

\begin{proof}
Let $\VV(I)$ denote the zero-set germ of $I$.
By the complex analytic Nullstellensatz (see, e.g., \cite[Ch.\,3, \S\,4.1]{loj}), $\VV(I)=\{p\}$ if and only if $\sqrt{I}=\mm_p$, or equivalently (by Noetherianity of $\OO_p$), $I$ contains a power of the maximal ideal $\mm_p$. Hence $\VV(I)$ equals $\{p\}$ precisely when both $K(I)$ and $D(I)$ are finite. On the other hand, $\VV(I)\varsupsetneq\{p\}$ if and only if there exists a $1$-dimensional irreducible complex-analytic germ $Y_p$ at $p$ such that every $g\in I$ vanishes on $Y_p$. Choosing $\gamma\in\Hp$ the Puiseux parametrization of $Y_p$ (see \cite[Ch.\,II, \S\,6.2]{loj}), we see that the latter is equivalent to $g\circ\gamma=0$ for every $g\in I$, that is, $\tau^*(I)=\infty$.

Assume then that $\VV(I)=\{p\}$, or equivalently, that $I$ contains a power of the maximal ideal $\mm_p$.
Observe that $I\subset J$ implies $\tau^*(I)\geq\tau^*(J)$. Hence, if $I\supset\mm_p^k$, then $\tau^*(I)\leq\tau^*(\mm_p^k)$. The inequality $\tau^*(I)\leq K(I)$ thus follows from the fact that $\tau^*(\mm_p^K)=K$ (as $\mm_p^K$ can be generated by monomials, all of degree $K$).

Suppose now that $\mm_p^k\not\subset I$. Then there is a multi-index $\beta\in\N^n$ of length $|\beta|=k$, such that $(z-p)^{\beta}\notin I$. It follows that $(z-p)^{\alpha}\notin I$ for every $\alpha=(\alpha_1,\dots,\alpha_n)\in\N^n$ satisfying $\alpha_j\leq\beta_j$, $j=1,\dots,n$. Since there is at least $|\beta|+1=k+1$ of such $\alpha$'s, then $\OO_p/I$ contains at least $k+1$ elements linearly independent over $\C$. This proves the inequality $K(I)\leq D(I)$.
\end{proof}

\subsection{The type of a real analytic principal ideal}

Let $\OO^{\R}_p={_n\OO^{\R}_p}$ denote the ring of real-valued real analytic germs at a point $p=(p_1,\dots,p_n)\in\C^n$.
Let $\displaystyle{\vr(z,\bar{z})=\sum_{\alpha,\beta\in\N^n}c_{\alpha\beta}(z-p)^{\alpha}(\overline{z-p})^{\beta}}$ be a power series representation of $\vr(z,\bar{z})\in\OO^{\R}_p$, convergent in an open neighbourhood of $p$ in $\C^n$. We define the \emph{type of $\vr$ at $p$} as
\begin{equation}
\label{eq:type}
\D(\vr,p)=\sup_{\gamma\in\Hp}\,\frac{\nu(\vr\circ\gamma)}{\nu(\gamma)}\,,
\end{equation}
where the order of vanishing is taken with respect to the maximal ideal 
$(\mathrm{Re}(\zeta),\mathrm{Im}(\zeta))$ of the ring $\R\{\mathrm{Re}(\zeta),\mathrm{Im}(\zeta)\}$ of real analytic germs at $0$ in $\C\cong\R^2$.
It is readily seen that $\D(u\!\cdot\!\vr,p)=\D(\vr,p)$ for any invertible $u\in\OO^{\R}_p$. Hence, since $\OO^{\R}_p$ is a UFD, we may speak of the \emph{type $\D(I,p)$ of a principal ideal $I=(\vr)$} in $\OO^{\R}_p$.

Let $X$ be a smooth real analytic hypersurface in an open neighbourhood $U$ of a point $p$ in $\C^n$. Then, after shrinking $U$ if necessary, there is a unique (up to multiplication by an invertible $u\in\OO^{\R}_p$) real analytic $\vr\in\OO^{\R}_p$ with $d\vr(p)\neq0$ and $X=\{z\in U:\vr(z,\bar{z})=0\}$. One defines (see \cite[Def.\,2.16]{da1}, \cite[\S\,3.3.3]{da2}) the \emph{type of $X$ at $p$} as $\D(X,p):=\D(\vr,p)$.
However, the type of a real analytic set $X$ is not well-defined if $X$ is not a hypersurface. Indeed, if the real codimension of $X$ at $p$ is greater than $1$, there is no canonical choice of a single defining function, and given two distinct defining functions $\vr_1, \vr_2$ for $X$ in a neighbourhood of $p$ there need not exist an invertible $u$ with $\vr_2=u\cdot\vr_1$. Consequently, the family of ideals $I(\vr,U,p)$ associated to $X_p$ (see below) is not an invariant of $X_p$, but only of the principal ideal $(\vr)\cdot\OO^{\R}_p$. (Thus D'Angelo's \cite[\S\,3.3.2, Prop.\,5]{da2} only applies to smooth real hypersurfaces.)
Nonetheless, we can state the following:

\begin{definition}
\label{def:fin-type}
Let $X$ be a closed real analytic subset of an open set in $\C^n$, and let $\vr(z,\bar{z})$ be any real analytic function in a neighbourhood $U$ of a point $p\in X$ satisfying $X\cap U=\{z\in U:\vr(z,\bar{z})=0\}$. We say that $p$ is a \emph{point of finite type of $X$}, when $\D(\vr,p)<\infty$. Otherwise, $p$ is called a \emph{point of infinite type of $X$}.
\end{definition}

\begin{remark}
\label{rem:well-def-type}
By Proposition \ref{prop:inf-type} below, the notion of a point of finite type is well-defined; i.e., independent of the choice of a defining function. Indeed, if $\vr_1$ and $\vr_2$ are two real analytic functions defining $X$ in a neighbourhood of a point $p\in X$, then $\D(\vr_1,p)=\infty$ iff $\D(\vr_2,p)=\infty$, because both equalities are equivalent to $X_p$ containing a positive-dimensional complex analytic germ.
\end{remark}

\subsection{Holomorphic decomposition}

Consider $\displaystyle{\vr(z,\bar{z})=\sum_{|\alpha|+|\beta|\geq1}c_{\alpha\beta}(z-p)^{\alpha}(\overline{z-p})^{\beta}}$
a real analytic function vanishing at $p$, with the power series convergent in the polydisc $\{z:|z_j-p_j|<\delta_j\}$. Let $\delta=(\delta_1,\dots,\delta_n)$, and let $0<t<1$. One can associate to $\vr$ functions
\begin{align*}
h(z)&=4\sum_{|\alpha|\geq1}c_{\alpha0}(z-p)^{\alpha}\,,\\
f^{\beta}(z)&=\sum_{|\alpha|\geq1}c_{\alpha\beta}(t\delta)^{\beta}(z-p)^{\alpha}+(z-p)^{\beta}(t\delta)^{-\beta}\,,\\
g^{\beta}(z)&=\sum_{|\alpha|\geq1}c_{\alpha\beta}(t\delta)^{\beta}(z-p)^{\alpha}-(z-p)^{\beta}(t\delta)^{-\beta}\,,
\end{align*}
for all $\beta\in\N^n$, $|\beta|\geq1$. It is easy to see that $h(z)$ and all the $f^{\beta}(z)$, $g^{\beta}(z)$ are holomorphic in the polydisc 
$\{z:|z_j-p_j|<t\delta_j\}$, and that 
$\left\|f(z)\right\|^2=\sum_{|\beta|\geq1}|f^{\beta}(z)|^2$,  $\left\|g(z)\right\|^2=\sum_{|\beta|\geq1}|g^{\beta}(z)|^2$ 
are real analytic in the same polydisc. One may thus consider $f=(f^{\beta})_{|\beta|\geq1}$ and $g=(g^{\beta})_{|\beta|\geq1}$ 
as holomorphic functions with values in the Hilbert space $l^2$.
Moreover, $\vr$ admits a \emph{holomorphic decomposition} of the form
\begin{equation}
\label{eq:holo-decomp}
4\vr(z,\bar{z})=2\mathrm{Re}(h(z))+\left\|f(z)\right\|^2-\left\|g(z)\right\|^2\,.
\end{equation}
For a unitary transformation $U:l^2\to l^2$, consider an ideal $I(\vr,U,p)$ in $\OO_p$ generated by $h(z)$ and by the components 
$\displaystyle{f^{\beta}(z)-\sum_{\sigma\in\N^n}u_{\beta\sigma}g^{\sigma}(z)}$ of $\,f-U(g)$, where $u_{\beta\sigma}$ are the entries of the (matrix of) $U$.

\begin{lemma}[{cf. \cite[Thm.\,3.5]{da1}}]
\label{lem:key-ineq}
The following inequality holds
\[
\D(\vr,p)\ \leq\ 2\,\sup_U\tau^*(I(\vr,U,p))\,,
\]
where the supremum is taken over all unitary transformations $U:l^2\to l^2$.
\end{lemma}

\begin{proof}
Suppose that $\gamma\in\Hp$ is such that $\nu(\vr\circ\gamma)>2k$ for some integer $k\geq1$. It suffices to find a unitary $U:l^2\to l^2$ for which $\tau^*(I(\vr,U,p))>k/\nu(\gamma)$. We have $J^{2k}(\vr\circ\gamma)=0$, where, for a germ $f\in\R\{x,y\}$, $J^s(f)$ denotes the \emph{$s$-jet} of $f$, that is, the image of $f$ under the homomorphism $J^s:\R\{x,y\}\to\R\{x,y\}/(x,y)^{s+1}$ of $\R\{x,y\}$-modules.
For simplicity of notation assume $p=0$. Then
\[
\vr(\gamma(\zeta),\overline{\gamma(\zeta)})=(\sum_{|\alpha|\geq1}c_{\alpha 0}\,\gamma(\zeta)^{\alpha} + \sum_{|\beta|\geq1}c_{0 \beta}\,\overline{\gamma(\zeta)}^{\beta})+\sum_{|\alpha|,|\beta|\geq1}c_{\alpha \beta}\,\gamma(\zeta)^{\alpha}\overline{\gamma(\zeta)}^{\beta}\,.
\]
Since the bracket on the right hand side of this equation contains only pure terms and all the other (non-zero) terms contain positive powers of both $\zeta$ and $\bar{\zeta}$, it follows from $J^{2k}(\vr\circ\gamma)=0$ that the $2k$'th jet of the bracket is zero. The content of the bracket is precisely $2\mathrm{Re}(h\circ\gamma)$, hence $J^{2k}(h\circ\gamma)=0$, and consequently $J^{2k}(\left\|f\circ\gamma\right\|^2-\left\|g\circ\gamma\right\|^2)=0$, by \eqref{eq:holo-decomp}. One checks by direct computation that the latter implies $\left\|J^k(f\circ\gamma)\right\|^2\ =\ \left\|J^k(g\circ\gamma)\right\|^2$.
Then, by Lemma~\ref{lem:eq-norms} below, there is a unitary $U:l^2\to l^2$ such that $J^k(f\circ\gamma)-U(J^k(g\circ\gamma))=0$. Since $J^k(f\circ\gamma)-U(J^k(g\circ\gamma))=J^k[(f-U(g))\circ\gamma]$, it follows that $\nu((f^{\beta}-\sum_{\sigma\in\N^n}u_{\beta\sigma}g^{\sigma})\circ\gamma)>k$ for all $|\beta|\geq1$. Therefore $\nu(F\circ\gamma)>k$ for every generator $F$ of $I(\vr,U,p)$, which proves $\tau^*(I(\vr,U,p))> k/\nu(\gamma)$.
\end{proof}

\begin{lemma}[{cf. \cite[\S\,3.3.1, Prop.\,4]{da2}}]
\label{lem:eq-norms}
Let $F,G:B\to l^2$ be holomorphic mappings on an open ball in $\C^q$, with $\left\|F\right\|^2=\left\|G\right\|^2$. Suppose there exists $k\in\N$ such that all the components of $F$ and $G$ are polynomials of degree at most $k$. Then there is a unitary operator $U:l^2\to l^2$ satisfying $F=U(G)$.
\end{lemma}

\begin{proof}
Write $F=\sum F_{\alpha}z^{\alpha}$, $G=\sum G_{\alpha}z^{\alpha}$. By expanding and equating the norms squared, one obtains relations 
\begin{equation}
\label{eq:a-b}
(F_{\alpha},F_{\beta})=(G_{\alpha},G_{\beta})
\end{equation}
for all multi-indices $\alpha,\beta$, where $(.,.)$ denotes the inner product in $l^2$.
Since all the components of $F$ and $G$ are polynomials of degree at most $k$, it follows that $\mathrm{span}(F_{\alpha})$ and $\mathrm{span}(G_{\alpha})$ are finite-dimensional vector spaces. Moreover, by \eqref{eq:a-b}, they are of the same dimension. Hence one can define $U:\mathrm{span}(G_{\alpha})\to\mathrm{span}(F_{\alpha})$ by setting $U(G_{\alpha})=F_{\alpha}$ on a maximal linearly independent set.
Then $U$ is a well-defined linear transformation and an isometry from $\mathrm{span}(G_{\alpha})$ to $\mathrm{span}(F_{\alpha})$. By defining $U$ to be an isometry from the orthogonal complement of $\mathrm{span}(G_{\alpha})$ to the orthogonal complement of $\mathrm{span}(F_{\alpha})$, one obtains an operator with the required properties.
\end{proof}

\begin{remark}
\label{rem:mistake}
We are indebted to the anonymous referee for pointing out a mistake in an earlier version of the above lemma.
In fact, the mistake can be traced back to \cite[\S\,3.3.1, Prop.\,4]{da2}, where it is not assumed that $F$ and $G$ are polynomial. As it turns out, without this assumption one cannot guarantee that the dimensions of $\mathrm{span}(F_{\alpha})$ and $\mathrm{span}(G_{\alpha})$ are the same \emph{and} so are the dimensions of their ortogonal complements. This can be seen readily if one sets, for example, $F=(z,z^2,z^3,\dots)$ and $G=(0,z,z^2,z^3,\dots)$.
In the general case, one can still prove that there exists an isometry $U$ (but not necessarily unitary) such that $F\oplus 0=U(G)$ or
$G\oplus 0=U(F)$, where $0$ is a certain (possibly infinite) vector of zeros.
\end{remark}

\subsection{The equivalence}

\begin{proposition}[{cf. \cite[\S\,3.3.3, Thm.\,4]{da2}}]
\label{prop:inf-type}
Let $X$ be a closed real analytic subset of an open set in $\C^n$, defined in a neighbourhood of a point $p\in X$ by the vanishing of a real analytic function $\vr(z,\bar{z})=\sum_{\alpha,\beta}c_{\alpha\beta}\,(z-p)^{\alpha}(\overline{z-p})^{\beta}$. Then $\D(\vr,p)<\infty$ if and only if the germ $X_p$ contains no positive-dimensional complex analytic germ.
\end{proposition}

\begin{proof} We follow the argument of Lempert~\cite{lem}.
Suppose $X_p$ contains a $1$-dimensional complex-analytic germ $Y_p$. Choosing $\gamma\in\Hp$ the Puiseux parametrization of (an irreducible component of) $Y$ at $p$, we get $\vr\circ\gamma=0$, hence $\D(\vr,p)=\infty$.

Conversely, assume that $X_p$ contains no positive-dimensional complex germs and, for a proof by contradiction, suppose that $\D(\vr,p)=\infty$. Then, by Lemma~\ref{lem:key-ineq}, there exists a sequence $(U^{j})_{j\geq1}$ of unitary matrices for which
\begin{equation}
\label{eq:lim-tau}
\lim_{j\to\infty}\tau^*(I(\vr,U^{j},p))=\infty\,.
\end{equation}
Denoting by $(U^{j})^*$ the adjoint of $U^{j}$, we have, for every $j$,
\begin{equation}
\label{eq:Uj}
I(\vr,U^{j},p)\ =\ (h,f-U^{j}(g))\cdot\OO_p\ =\ (h,f-U^{j}(g),(U^{j})^*(f)-g)\cdot\OO_p\,,
\end{equation}
since $(U^j)^*=(U^j)^{-1}$ and ideals in $\OO_p$ are closed in the topology of coefficient-wise convergence (see, e.g., \cite[Thm.\,6.3.5]{hor}).
By \eqref{eq:lim-tau} and Lemma~\ref{lem:ineq1}, it follows that
\begin{equation}
\label{eq:Uj-infty}
\lim_{j\to\infty}D(I(\vr,U^{j},p))=\infty\,.
\end{equation}

The entries $u^j_{\beta\sigma}$ of every $U^{j}$ with respect to any complete orthonormal set are bounded in absolute value by $1$.
Hence, it can be assumed that, for all $\beta,\sigma\in\N^n$, the sequence $(u^j_{\beta\sigma})_{j\geq1}$ has a limit, say, $u^\infty_{\beta\sigma}$. Denote by $U^\infty$ the limit operator $(u^\infty_{\beta\sigma})$, and let $(U^\infty)^*=(\tilde{u}^\infty_{\beta\sigma})$ denote its adjoint.

Let $Y_p$ be the zero-set germ of the ideal $J=(h,f-U^{\infty}(g),(U^{\infty})^*(f)-g)\cdot\OO_p$. The operator norms of $U^\infty$ and of $(U^\infty)^*$ are less than or equal to $1$ (however, $U^\infty$ need not be unitary). Therefore, for every $z$ in a (sufficiently small) representative of $Y_p$, we have
\[
\left\|f(z)\right\|\ =\ \left\|U^{\infty}(g(z))\right\|\ \leq\ \left\|g(z)\right\|\ =\ \left\|(U^{\infty})^*(f(z))\right\|\ \leq\ \left\|f(z)\right\|\,.
\]
Thus $Y_p\subset X_p$, by \eqref{eq:holo-decomp}, and hence $Y_p$ is the germ of the singleton $\{p\}$, by assumption. Consequently $D(J)<\infty$, by Lemma~\ref{lem:ineq1}; say, $D(J)=d$.

Now, by noetherianity of $\OO_p$, there exists $N\in\N$ such that
\[
J=(h,\, f^{\beta}-\sum_\sigma u_{\beta\sigma}g^{\sigma},\, g^{\beta}-\sum_\sigma \tilde{u}_{\beta\sigma}f^{\sigma}\ :\ |\beta|\leq N)\,.
\]
Set
\[
I_j=(h,\, f^{\beta}-\sum_\sigma u^j_{\beta\sigma}g^{\sigma},\, g^{\beta}-\sum_\sigma \tilde{u}^j_{\beta\sigma}f^{\sigma}\ :\ |\beta|\leq N)\,,
\]
where $\tilde{u}^j_{\beta\sigma}$ are the entries of $(U^j)^*$. By the Banach-Steinhaus Theorem, in a sufficiently small neighbourhood of $p$, all $f^{\beta}-\sum_\sigma u^j_{\beta\sigma}g^{\sigma}$ (resp. $g^{\beta}-\sum_\sigma \tilde{u}^j_{\beta\sigma}f^{\sigma}$) converge uniformly to $f^{\beta}-\sum_\sigma u_{\beta\sigma}g^{\sigma}$ (resp. to $g^{\beta}-\sum_\sigma \tilde{u}_{\beta\sigma}f^{\sigma}$) as $j\to\infty$. Hence, by the upper semi-continuity of $D(I)=\dim_{\C}\OO_p/I$ as a function of $I$ (\cite[Ch.\,II, Prop.\,5.3]{tou}),
we have $D(I_j)\leq d$ for $j$ large enough. On the other hand,
\[
I_j\ \subset\ (h,\, f^{\beta}-\sum_\sigma u^j_{\beta\sigma}g^{\sigma},\, g^{\beta}-\sum_\sigma \tilde{u}^j_{\beta\sigma}f^{\sigma}\ :\ \beta\in\N^n)\ =\ I(\vr,U^{j},p)\,,
\]
where the equality follows from \eqref{eq:Uj}. Therefore $D(I(\vr,U^{j},p))\leq D(I_j)\leq d$, which contradicts \eqref{eq:Uj-infty}.
\end{proof}


\end{document}